\documentclass[a4paper,11pt]{article}
\usepackage{inputenc}
\usepackage{authblk}
\usepackage[table]{xcolor}
\usepackage{color}
\usepackage{amsmath}
\usepackage{amssymb}
\usepackage{amsthm}
\usepackage{booktabs}
\usepackage{xcolor,graphicx,float}
\usepackage{hyperref}
\usepackage{soul}
\usepackage[numbers,sort&compress]{natbib}
\usepackage{indentfirst}
\hypersetup{colorlinks=true,
	linkcolor=blue,
	anchorcolor=blue,
	citecolor=blue}
\usepackage{enumitem}
\setlist[itemize]{itemsep=0pt, topsep=0pt}

\theoremstyle{plain}
\newtheorem{theorem}{\bf Theorem}[section]
\newtheorem{lemma}[theorem]{\bf Lemma}
\newtheorem{proposition}[theorem]{\bf Proposition}
\newtheorem{corollary}[theorem]{\bf Corollary}

\theoremstyle{remark}

\newtheorem{remark}{\bf Remark}
\newtheorem{construction}{\bf Construction}
\newtheorem{assumption}{\bf Assumption}

\numberwithin{equation}{section}

\newcommand{\spn}[2]{\genfrac{\{}{\}}{0pt}{}{#1}{#2}}
\title{\bf\Large Erd\H{o}s-Ko-Rado theorem and Hilton-Milner type theorem for $k$-partitions}
\author[1]{Jie Wen\thanks{E-mail: \text{jwen@mail.bnu.edu.cn}}}
\author[1]{Benjian Lv\thanks{Corresponding author. E-mail: \text{bjlv@bnu.edu.cn}}}
\affil[1]{\small Laboratory of Mathematics and Complex Systems (Ministry of Education), School of Mathematical Sciences, Beijing Normal University, Beijing 100875, China}
\date{}

\begin{document}
	\renewcommand{\baselinestretch}{1.2}
	\maketitle
	\begin{abstract}
	A $k$-partition of an $n$-set $X$ is a collection of $k$ pairwise disjoint non-empty subsets whose union is $X$. A family of $k$-partitions of $X$ is called $t$-intersecting if any two of its members share at least $t$ blocks. A $t$-intersecting family is trivial if every $k$-partition in it contains $t$ fixed blocks, and is non-trivial otherwise. In this paper, we first prove that, for $n\geq L(k,t):=(t+1)+(k-t+1)\cdot\log_2(t+1)(k-t+1)$, a $t$-intersecting family with maximum size must consist of all $k$-partitions containing $t$ fixed singletons. This improves the results given by Erd\H{o}s and Sz\'{e}kely (2000), and by Kupavskii (2023). We further determine the non-trivial $t$-intersecting families of $k$-partitions with maximum size for $n \ge 2L(k,t)$, which turn out to be natural analogs of the corresponding families for finite sets. 
In addition, we prove a stability result.

		
		
	\end{abstract}
	\section{Introduction}
	For
	the sake of convenience, lowercase letters denote positive integers, unless otherwise stated. We use $[n]=\{1,2,\ldots,n\}$ to denote the standard $n$-set, and set $[i,j]=\{i,i+1,\ldots,j\}$ for $i\leq j$. A \emph{family} is just a collection of sets. For a family $\mathcal{F}$, we denote $\cup\mathcal{F}:=\cup_{F\in\mathcal{F}}F$ and $\cap\mathcal{F}:=\cap_{F\in\mathcal{F}}F$ for short. Denote by $\binom{[n]}{k}$ the family of all $k$-subsets of $[n]$. A family $\mathcal{F}\subseteq\binom{[n]}{k}$ is \emph{$t$-intersecting} if $|F\cap F'|\geq t$ for all $F,F'\in\mathcal{F}$. We write  $\overline{A}=[n]\setminus A$ for $A\subseteq[n]$.
	
	The seminal Erdős–Ko–Rado theorem \cite{Erdos-Ko-Rado-1961} initiated the study of intersection problems, which has ever since remained a significant area within extremal combinatorics. The theorem  states that, if $n$ is sufficiently large depending on $k$ and $t$, then every $t$-intersecting family of  $\binom{[n]}{k}$ has size at most $\binom{n-t}{k-t}$, and every family achieving the maximum must consist of all $k$-subsets which contain some fixed $t$-subset. Denote by $n_0(k,t)$ the least possible value for $n$ such that the upper bound in the Erd\H{o}s-Ko-Rado theorem holds. The original paper \cite{Erdos-Ko-Rado-1961} proved $n_0(k,1)=2k$, and established $n_0(k,t)\leq(k-t)\binom{k}{t}^3+t$. Frankl \cite{Frankl-1976} made a breakthrough by determining $n_0(k,t)=(t+1)(k-t+1)$ for $t\geq 15$. Wilson \cite{Wilson-1984} subsequently showed $n_0(k,t)=(t+1)(k-t+1)$ for $2\leq t\leq14$, via an elegant algebraic proof valid for all $t\geq1$. In \cite{Frankl-1976}, Frankl made a general conjecture on the maximum size of $t$-intersecting families of $\binom{[n]}{k}$. This conjecture was proved for $n>c\sqrt{t/\log t}(k-t+1)$ by Frankl and F{\"u}redi \cite{Frankl--Furedi-1991}. Ahlswede and Khachatrian \cite{Ahlswede-Khachatrian-1997} completely settled the problem by establishing the famous complete intersection theorem, which determines $t$-intersecting families of $\binom{[n]}{k}$ with maximum size for all values of $n, k$, and $t$.
	
	A $t$-intersecting family $\mathcal{F}\subseteq\binom{[n]}{k}$ is \emph{trivial} if $|\cap\mathcal{F}|\geq t$, and is \emph{non-trivial} otherwise. The family $\mathcal{F}$ is \emph{maximal} if $\mathcal{F}\cup\{A\}$ is not $t$-intersecting for any $A\in\binom{[n]}{k}\setminus\mathcal{F}$. Note that every optimal family in the Erd\H{o}s-Ko-Rado theorem is trivial. The classical Hilton-Milner theorem \cite{Hilton-Milner-1967} gives the maximum size of non-trivial $1$-intersecting families, with a characterization of extremal configurations. A significant step was taken in \cite{Frankl-1978} by Frankl, who determined non-trivial $t$-intersecting families with maximum size for $t\geq2$ and large $n$. Ahlswede and Khachatrian \cite{Ahlswede-Khachatrian-1996} determined such families for all possible parameters. Then it is natural to study the structure of large non-trivial $t$-intersecting families. The problem has been well explored, yielding several insightful results, and these can be regarded as stability results for the three classical results mentioned above. Han and Kohayakawa \cite{Han-Kohayakawa} described the structure of the third largest maximal $1$-intersecting families, and proposed a question that what is the size of the fourth largest one. For results on this question, see \cite{Kostochka-Mubayi,Kupavskii-2025,Huang-Peng,Wu et al.,Kupavskii-2018}. Recently, Ge, Xu and Zhao \cite{Ge-Xu-Zhao} developed an unified approach to proving stability results. For general $t\geq1$, Cao, the second author and Wang \cite{Lv-2021} described the structure of maximal non-trivial $t$-intersecting families with nearly maximum size. 
	
The notion of intersection arises naturally for many objects, and it turns out that investigating intersection problems has given rise to a variety of elegant works. For systematic introduction on intersection problems, we refer the readers to survey papers \cite{Deza-Frankl-1983,Frankl-Tokushige-2016} and monographs \cite{Frankl-Tokushige-book,Godsil-Meagher-book}. 
	
 A \emph{partition} is a set of pairwise disjoint non-empty subsets. A member of a partition is said to be a \emph{block}. We say that a partition is a \emph{partition of $M$} if the union of its blocks equals $M$. A family of partitions of $[n]$ is said to be \emph{$t$-intersecting} if every two of its members has at least $t$ blocks in common. The intersection problem for partitions constitutes one of the \emph{high order extremal problems}, namely, extremal problems in which the set systems under consideration are collections of families. Ku and Renshaw \cite{Ku-Renshaw-2008}, Ku and Wong \cite{Ku-Wong-2013,Ku-Wong-2015,Ku-Wong-2020}, and Blinovsky \cite{B} established several intersection theorems for partitions of $[n]$. Meagher and Moura \cite{Meagher-Moura-2005} proved an Erd\H{o}s-Ko-Rado theorem for partitions of  $[ck]$ whose every block has size $c$. See \cite{Godsil-Meagher-2017,Lindzey-2017} for an algebraic approach valid when $c=2$.

 In this paper, we focus on the intersection problem for \emph{$k$-partitions}, that is, partitions with $k$ blocks. We use $\spn{[n]}{k}$ to denote the collection of such partitions of $[n]$. It is well known that the {Stirling partition number}
	  \begin{equation}\label{equspn}
	  	\spn{n}{k}=\frac{1}{k!}\sum_{j=0}^{k}(-1)^j\binom{k}{j}(k-j)^n
	  \end{equation}
	  counts the size of $\spn{[n]}{k}$. Set $\spn{n}{a}=0$ for $a\leq0$. Let $\mathcal{F}\subseteq\spn{[n]}{k}$ be a $t$-intersecting family. It is \emph{trivial} if every member of it contains a common specified $t$-partition and is \emph{non-trivial} otherwise. It is \emph{maximal} if $\mathcal{F}=\mathcal{G}$ for each $t$-intersecting family $\mathcal{G}\subseteq\spn{[n]}{k}$ with $\mathcal{F}\subseteq\mathcal{G}$. It is routine to check that the maximal trivial $t$-intersecting families are exactly those given in the following  construction.
	  \begin{construction}\label{familyf}
		Let $X$ be a $t$-partition. Define
		\begin{align*}
			\mathcal{S}(X):=\left\{F\in\spn{[n]}{k}:X\subseteq F\right\}.
		\end{align*}
Clearly,
\begin{equation}\label{equntkt}
	|\mathcal{S}(X)|=\spn{n-|\cup X|}{k-t}\leq\spn{n-t}{k-t},
\end{equation}
with equality precisely if $|\cup X|=t$, namely, $X$ consists of singletons.	
	\end{construction}
	
	In 2000, P. L. Erd\H{o}s and Sz\'{e}kely \cite{Erdos-Szekely-2000} proved a $k$-partition version of Erd\H{o}s-Ko-Rado theorem. Precisely, if $k>t\geq1$ and $n$ is large enough with respect to $k$, then every $t$-intersecting family in $\spn{[n]}{k}$ has at most $\spn{n-t}{k-t}$ members. Recently, Kupavskii \cite{Kupavskii-2023} applied the spread approximation method \cite{Kupavskii-2024} to address Erd\H{o}s-Ko-Rado type questions on several classes of partitions, including proving that, for $n\geq2k\log_2n$ and $n\geq48$, every maximum-sized $t$-intersecting family in $\spn{[n]}{k}$ consists of all $k$-partitions with $t$ fixed singletons. Our first main result provide an improved lower bound for $n$. Clearly, two distinct $(t+1)$-partitions of $[n]$ share at most $t-1$ blocks. Hence let us suppose $k\geq t+2$. Set
	\begin{equation}\label{equfunlkt}
		L(k,t)=(t+1)+(k-t+1)\cdot\log_2(t+1)(k-t+1).
	\end{equation}
	Note that $L(k,t)\leq L(k,1)=2+k+k\log_2k$.  
	\begin{theorem}\label{thm1}
		Let $k\geq t+2$ and $n\geq L(k,t)$. If  $\mathcal{F}\subseteq\spn{[n]}{k}$ is $t$-intersecting, then
		$$|\mathcal{F}|\leq\spn{n-t}{k-t}.$$
		Equality holds only if $\mathcal{F}$ consists of all $k$-partitions with $t$ fixed singletons.  
	\end{theorem}
We will see in Remark \ref{rmk} that, if $k\sim(1+a)t$ as $t\to\infty$ for some fixed $a>0$, then the least value of $n$ for which the theorem above holds equals $\Theta(L(k,t))$. Our remaining main results focus on the structure of large non-trivial $t$-intersecting families. The key tool is the notion of a $t$-cover. Let $\mathcal{P}$ be a family of partitions. A \emph{$t$-cover} of $\mathcal{P}$ is a partition $T$ such that $|T\cap P|\geq t$ for each $P\in\mathcal{P}$, and if such a partition exists, then the \emph{$t$-covering number} $\tau_t(\mathcal{P})$ of $\mathcal{P}$ is defined to be the minimum size of a $t$-cover. Of course $\tau_t(\mathcal{P})\geq t$. We note that a $t$-cover is not required to be a partition of the universe set. Let $\mathcal{F}\subseteq\spn{[n]}{k}$ be a $t$-intersecting family, then every member of $\mathcal{F}$ itself forms a $t$-cover, and so $\tau_t(\mathcal{F})\leq k$. The family $\mathcal{F}$ is non-trivial precisely if it has $t$-covering number at least $t+1$. Let $B\subseteq[n]$, denote by $$[B]:=\{\{i\}:i\in B\}$$ the partition of $B$ containing only singletons. Let us introduce the following two structures of non-trivial $t$-intersecting families.
  \begin{construction}\label{familya}
  	Let $Z$ be a $(t+2)$-partition. Define
  	\begin{align*}
  		\mathcal{A}(Z):=\left\{F\in\spn{[n]}{k}:|F\cap Z|\geq t+1\right\}.
  	\end{align*}
  	Set $\mathcal{A}(n,k,t):=\mathcal{A}([[t+2]])$.
  \end{construction}
  \begin{construction}\label{familyh}
  	Let $M$ be a $k$-partition and let $X\subseteq M$ with $|X|=t$. Define
  	\begin{align*}
  		\mathcal{H}(X,M):=&\left\{F\in\spn{[n]}{k}:X\subseteq F,\;F\cap(M\setminus X)\neq\emptyset\right\}\nonumber\\
  		&\cup\{(M\setminus\{B\})\cup\{B\cup\left(\overline{\cup M}\right)\}:B\in X\}.
  	\end{align*}
  	Set $\mathcal{H}(n,k,t):=\mathcal{H}([[t]],[[k]])$ and $\mathcal{H}_1(n,k,t):=\mathcal{H}([[t]],[[k-1]]\cup\{\{k,k+1\}\})$.
  \end{construction}
 It is readily seen that each family confirming to one of the two  structures is of $t$-covering number $t+1$. The structure of non-trivial $t$-intersecting families in $\spn{[n]}{t+2}$ is quite clear. We will see in Proposition \ref{propt+2} that, if $k=t+2$ and  $\mathcal{F}\subseteq\spn{[n]}{k}$ is a maximal non-trivial $t$-intersecting family, then $\mathcal{F}=\mathcal{A}(Z)$ for some $(t+2)$-partition $Z$, and particularly $\mathcal{F}$ has exactly $t+2$ members.  Therefore, let us suppose $k\geq t+3$. 
 To present our results, we say that two families $\mathcal{G}_1,\mathcal{G}_2\subseteq\spn{[n]}{k}$ are \emph{isomorphic}, denoted $\mathcal{G}_1\cong\mathcal{G}_2$, if they are the same up to some permutation on $[n]$. The following result determines non-trivial $t$-intersecting families with maximum size.
\begin{theorem}\label{thm2}
		Let $k\geq t+3$ and $n\geq2L(k,t)$. Suppose that  $\mathcal{F}\subseteq\spn{[n]}{k}$ is a non-trivial $t$-intersecting family. Then
		$$|\mathcal{F}|\leq\max\{|\mathcal{A}(n,k,t)|,|\mathcal{H}(n,k,t)|\}.$$
		Moreover, if equality holds, then
		\begin{itemize}
			\item[\rm(i)] $\mathcal{F}\cong\mathcal{H}(n,k,t)$ for $k\geq2t+3$, or
			\item[\rm(ii)] $\mathcal{F}\cong\mathcal{A}(n,k,t)$ for $k\leq 2t+2$, or $(k,t)=(4,1)$ and $\mathcal{F}\cong\mathcal{H}(n,4,1)$.
		\end{itemize}
	\end{theorem}
	One notable feature of this result might be that our extremal families $\mathcal{H}(n,k,t)$ and $\mathcal{A}(n,k,t)$ serve as natural analogs of those in the classical results given in \cite{Hilton-Milner-1967} and \cite{Frankl-1978}. 
Let $n\geq2L(k,t)$ and $\mathcal{F}\subseteq\spn{[n]}{k}$ be a maximal $t$-intersecting family with size smaller than $\spn{n-t}{k-t}$. As an immediate consequence of Theorem \ref{thm2}(see Corollary \ref{coro}), we have  $|\mathcal{F}|\leq\spn{n-t-1}{k-t}<\frac{1}{k-t}\spn{n-t}{k-t}$. We remark that this gives a slightly stronger version of a result in \cite{Kupavskii-2023}, in which it is shown that $|\mathcal{F}|\leq\frac{1}{2}\spn{n-t}{k-t}$. This also serves as a stability result of Theorem \ref{thm1} for $n\geq2L(k,t)$. Our next main result is a stability result for Theorem \ref{thm2}.
\begin{theorem}\label{thm3}
	Let $k\geq t+3$ and $n\geq2L(k,t)$, and let   $\mathcal{F}\subseteq\spn{[n]}{k}$ be a maximal non-trivial $t$-intersecting family. The following hold.
	\begin{itemize}
		\item[\rm(i)]Suppose $k\geq2t+3$. If   $|\mathcal{F}|\geq|\mathcal{H}(n,k,t)|-\spn{n-t-1}{k-t-1}+\spn{n-t-2}{k-t-1}$, then $\mathcal{F}$ is isomorphic to either $\mathcal{H}(n,k,t)$ or $\mathcal{H}_1(n,k,t)$.
		\item[\rm(ii)]Suppose $t+3\leq k\leq2t+2$. If $|\mathcal{F}|>|\mathcal{H}(n,k,t)|$, 
		then $\mathcal{F}=\mathcal{A}(Z)$ for some  $(t+2)$-partition $Z$.
	\end{itemize}
\end{theorem} 
Let us mention that, by Lemma \ref{lemmahh1v} (i), the families listed in Theorem \ref{thm3} (i) are exactly those with size at least $|\mathcal{H}(n,k,t)|-\spn{n-t-1}{k-t-1}+\spn{n-t-2}{k-t-1}$. On the other hand, for a $(t+2)$-partition $Z$, we have  $|\mathcal{A}(Z)|=\sum_{B\in Z}\spn{n-|\cup Z|+|B|}{k-t-1}-(t+1)\spn{n-|\cup Z|}{k-t-2}$, and it might be hard to list the partitions $Z$ achieving $|\mathcal{A}(Z)|>|\mathcal{H}(n,k,t)|$.

The rest of this paper is organized as follows. In Section \ref{section2}, we prove Theorem~\ref{thm1}, and in Section \ref{section3}, we prove Theorems \ref{thm2} and \ref{thm3}. In Section \ref{section4}, we prove several inequalities used in the paper.
	\section{Erd\H{o}s-Ko-Rado theorem}\label{section2}
In this section, we first prove some lemmas, in which a family is associated with its $t$-covers and quantitatively characterized. These enables us to prove Theorem \ref{thm1}. We then determine all maximal $t$-intersecting families in $\spn{[n]}{t+2}$ by proving Proposition \ref{propt+2}.

Let us list several relations concerning Stirling partition numbers. 
 The numbers satisfy the recurrence relation 
	\begin{equation}\label{equrecurrence}
		\spn{n}{k}=\spn{n-1}{k-1}+k\spn{n-1}{k}.
	\end{equation}
	In particular, this gives
	\begin{equation}\label{equnkn-1k}
		\spn{n}{k}>k\spn{n-1}{k}
	\end{equation}
for $n\geq k\geq2$. The following lemma is established in the proof of {\rm\cite[Lemma 19]{Kupavskii-2023}}.
\begin{lemma}[\cite{Kupavskii-2023}]\label{lemmainequntkt}
	For $n\geq k\geq 2$, we have 
	$$\spn{n}{k}\geq\left(2^{\frac{n-1}{k-1}}-1\right)\spn{n-1}{k-1}.$$
\end{lemma}
	Let $U$ be a set and let $\mathcal{F}$ be a family whose members are subsets of $U$. For $H\subseteq U$, set
	\begin{equation*}
		\mathcal{F}_H:=\{F\in\mathcal{F}:H\subseteq F\}.
	\end{equation*}
	\begin{lemma}\label{lemmainductive}
		Suppose $\mathcal{F}\subseteq\spn{[n]}{k}$ and $T$ is a $t$-cover of $\mathcal{F}$ with $\ell$ blocks. If $S$ is an $s$-partition with $|S\cap T|=r<t$, and $\mathcal{F}_S\neq\emptyset$, then there is a $(t+s-r)$-partition $H$ containing $S$ such that $|\mathcal{F}_S|\leq\binom{\ell-r}{t-r}|\mathcal{F}_H|$. In particular,  $|\mathcal{F}_S|\leq\binom{\ell-r}{t-r}\spn{n-s-t+r}{k-s-t+r}$.
	\end{lemma}
	\begin{proof}
		For every $F\in\mathcal{F}_S$, we have
	\begin{align*}
		|F\cap(S\cup T)|&=|F|+|S\cup T|-|F\cup T|\\
		&\geq k+(\ell+s-r)-(k+\ell-t)=t+s-r.
	\end{align*}
	It follows that $\mathcal{F}=\cup_{P\in\mathcal{H}}\mathcal{F}_P,$
	where $\mathcal{H}$ is the set of all $(t+s-r)$-partitions containing $S$ with blocks lying in $S\cup T$. Note that $|\mathcal{H}|\leq\binom{\ell-r}{t-r}$. Since $\mathcal{F}_S$ is non-empty, so is $\mathcal{H}$. Now the desired partition can be chosen from those $H\in\mathcal{H}$ for which   $|\mathcal{F}_H|\geq|\mathcal{F}_P|$ whenever $P\in\mathcal{H}$. The latter estimate then follows from (\ref{equntkt}).
	\end{proof}
	\begin{lemma}\label{lemmakey}
	Suppose $k\geq t+2$, $n\geq L(k,t)$ and $\mathcal{F}\subseteq\spn{[n]}{k}$. If $\mathcal{G}$ is a collection of $t$-covers of $\mathcal{F}$, where each has at most $k$ blocks, then for every $t$-partition $H$, we have 
		$$|\mathcal{F}_H|\leq\max\left\{(k-t+1)^{\tau_t(\mathcal{G})-t}\spn{n-\tau_t(\mathcal{G})}{k-\tau_t(\mathcal{G})},\;(k-t+1)^{k-t}\right\}.$$
	\end{lemma}
	\begin{proof}
	Note that every member of $\mathcal{F}$ also forms a $t$-cover of $\mathcal{G}$, and so $\tau_t(\mathcal{F})\leq k$. The lemma holds trivially for $\mathcal{F}_H=\emptyset$ or    $\tau_t(\mathcal{G})=t$. Now suppose $\mathcal{F}_H\neq\emptyset$ and $\tau_t(\mathcal{G})>t$. Then $H$ is not a $t$-cover of $\mathcal{G}$, which implies that $\dim(H\cap G_1)=r_1<t$ for some $G_1\in\mathcal{G}$. Set $H_1=H$. By Lemma \ref{lemmainductive}, we can inductively choose $G_1,G_2,\ldots,G_j\in\mathcal{G}$ and partitions $H_1\subsetneqq H_2\subsetneqq\cdots\subsetneqq H_j$ such that $|H_i\cap G_i| <t$, 
	\begin{equation}\label{equlemmakey1}
		|H_{i+1}|=|H_{i}|+t-|H_i\cap G_i|
	\end{equation}
	and
	\begin{equation}\label{equlemmakey2}
		|\mathcal{F}_{H_i}|\leq\binom{|G_i|-|H_i\cap G_i|}{t-|H_i\cap G_i|}|\mathcal{F}_{H_{i+1}}|
	\end{equation}   
	for $1\leq i\leq j-1$, and $|H_{j-1}|<\tau_t(\mathcal{G})\leq|H_j|$.
	Note that $|G_i|\leq k$ for $1\leq i\leq j$. Combining (\ref{equlemmakey1}), (\ref{equlemmakey2}) and Lemma \ref{lemmalrtr} yields 
	$$|\mathcal{F}_{H}|\leq\prod_{i=1}^{j-1}\binom{k-|H_i\cap G_i|}{t-|H_i\cap G_i|}|\mathcal{F}_{H_j}|\leq\mathcal(k-t+1)^{|H_j|-t}|\mathcal{F}_{H_j}|.$$
	If $|H_j|\geq k$, then $|H_j|=k$ and $H_j\in\mathcal{F}$ as $\mathcal{F}_H$ is non-empty, and thus $|\mathcal{F}_{H_j}|=1$, implying that $|\mathcal{F}_H|\leq(k-t+1)^{k-t}$. If $|H_j|\leq k-1$, then from (\ref{equntkt}) and Lemma \ref{lemmamono} (i), we conclude that $|\mathcal{F}_H|\leq (k-t+1)^{\tau_t(G)-t}\spn{n-\tau_t(G)}{k-\tau_t(G)}$. 
	\end{proof}
	For simplicity, we define
\begin{equation}\label{equfunf}
f(m,k,t,n)=\left\{\begin{array}{ll}(k-t+1)^{m-t}\binom{m}{t}\spn{n-m}{k-m},&{\rm if}\;t\leq m\leq k-1,\\(k-t+1)^{k-t}\binom{k}{t},&{\rm if}\;m=k.\\\end{array}\right.
\end{equation}

\noindent {\bf Proof of Theorem \ref{thm1}.}\;Since $\mathcal{F}$ is $t$-intersecting, $\mathcal{F}$ itself forms a set of $t$-covers. Let $T$ be a $t$-cover of $\mathcal{F}$ of size $\tau_t(\mathcal{F})$, then $\mathcal{F}=\cup_{H}\mathcal{F}_H$ with $H$ ranging over all $t$-subsets of $T$. From Lemma \ref{lemmakey} and  $t\leq\tau_t(\mathcal{F})\leq k$, we have
\begin{align}\label{equthm1}
	|\mathcal{F}|&\leq\max\{f(\tau_t(\mathcal{F}),k,t,n),\;f(k,k,t,n)\}.
\end{align}
It follows from Lemma \ref{lemmamono} (iii) that $|\mathcal{F}|\leq\spn{n-t}{k-t}$. Suppose equality holds, then by Lemma \ref{lemmamono} (iii) again, we obtain  $\tau_t(\mathcal{F})=t$, namely, every member of $\mathcal{F}$ shares some common $t$-partition, say $X$. From (\ref{equntkt}), we deduce that $\mathcal{F}=\left\{F\in\spn{[n]}{k}:X\subseteq F\right\}$ and $X$ consists of $t$ singletons. This completes the proof.\hfill{$\square$}\vspace{1em}

We conclude this section with the following proposition, which shows that the $t$-intersecting families in $\spn{[n]}{t+2}$ can be explicitly determined.
\begin{proposition}\label{propt+2}		Let $n>t+2$. If $\mathcal{F}\subseteq\spn{[n]}{t+2}$ is a maximal non-trivial $t$-intersecting family, then $\mathcal{F}=\mathcal{A}(Z)$ for some $(t+2)$-partition $Z$. \end{proposition} 
\begin{proof}
	Note that two $(t+2)$-partitions of $[n]$ meeting in at least $t+1$ blocks are the same. Then any two of members of $\mathcal{F}$ has exactly $t$ blocks in common. Fix two of them, say 
	$$F_1=G\cup\{C_1,C_2\},\;F_2=G\cup\{C_3,C_4\},$$
	where $G=F_1\cap F_2$, and $\{C_1,C_2\}$ and $\{C_3,C_4\}$ are necessarily distinct partitions of $\overline{\cup G}$. Note that $\mathcal{F}$ is non-trivial, then there exists $F_3\in\mathcal{F}$ with $|F_3\cap G|<t$. Since $|F_3\cap F_1|\geq t$, we have  $F_3\cap\{C_1,C_2\}\neq\emptyset$. Similarly, $F_3\cap\{C_3,C_4\}\neq\emptyset$. Any two blocks of $F_3$ are disjoint, from which we deduce that $C_i\cap C_j=\emptyset$ for some $i\in\{1,2\}$ and $j\in\{3,4\}$. Without loss of generality, suppose $C_1\cap C_4=\emptyset$. Then $C_1\subseteq C_3$ and $C_4\subseteq C_2$, and hence $C_i\cap C_j\neq\emptyset$ whenever $(i,j)\in\{(1,3),(2,3),(2,4)\}$. Therefore, $F_3\cap\{C_1,C_2\}=\{C_1\}$ and $F_3\cap\{C_3,C_4\}=\{C_4\}$, and so $|F_3\cap G|=t-1$. Set $G_0:=F_3\cap G$ and suppose $G\setminus G_0=\{C_0\}$, then 
	$$F_3=G_0\cup\{C_1,C_4\}\cup\{C_0\cup(C_2\cap C_3)\}.$$
	
	Now we claim $\mathcal{F}_{G}=\{F_1,F_2\}$. In fact, if $F\in\mathcal{F}$ and $G\subseteq F$, then $C_0\cup(C_2\cap C_3)$ cannot lie in $F$ as $C_0\in F$. Since $|F\cap F_3|\geq t$, one of $C_1$ and $C_4$ belongs to $F$ and so $F$ is either $F_1$ or $F_2$. This justifies our claim. Therefore, $|F\cap G|<t$ for every $F\in\mathcal{F}\setminus\{F_1,F_2\}$. By the same argument as we determined $F_3$, we obtain $\mathcal{F}\subseteq\mathcal{A}(Z)$, where $Z=G\cup\{C_1,C_4\}$. It follows from the maximality of $\mathcal{F}$ that $\mathcal{F}=\mathcal{A}(Z)$. 
\end{proof}
\section{Non-trivial $t$-intersecting families}\label{section3}
In this section, we characterize non-trivial $t$-intersecting families in $\spn{n}{k}$ using $t$-covers, where $k\geq t+3$ and $n\geq2L(k,t)$, and on this basis we prove Theorems \ref{thm2} and \ref{thm3}. Let us recall a useful result concerning inclusion-exclusion.
\begin{lemma}\label{lemmain-ex}
	Let $A_1,A_2,\ldots,A_m$ be finite sets. Then $$|\cup_{i\in[m]}A_i|\leq\textstyle\sum_{j\in[s]}(-1)^{j-1}\textstyle\sum_{J\in\binom{[m]}{j}}|\cap_{\ell\in J}A_\ell|$$
	for odd $s$, and 
	$$|\cup_{i\in[m]}A_i|\geq\textstyle\sum_{j\in[s]}(-1)^{j-1}\textstyle\sum_{J\in\binom{[m]}{j}}|\cap_{\ell\in J}A_\ell|$$
	for even $s$.
\end{lemma}
\begin{lemma}\label{lemmaspn2lkt}
	If $k\geq t+3$ and $n\geq2L(k,t)$, then for each integer $s$ with  $0\leq s\leq k-t-2$, we have 
	\begin{equation*}
		\spn{n-t-s}{k-t-s}>(t+1)^2(k-t+1)^2\spn{n-t-s-1}{k-t-s-1}.
	\end{equation*}
\end{lemma}
\begin{proof}
	Since  $\frac{2L(k,t)-t-1}{k-t-1}>\frac{t+1}{k-t-1}+2\log_2(t+1)(k-t+1)$,
	\begin{equation}\label{equlemmaspn2lkt}
	2^{\frac{n-(t+s)-1}{k-(t+s)-1}}\geq2^{\frac{n-t-1}{k-t-1}}>2^{\frac{t+1}{k-t-1}}(t+1)^2(k-t+1)^2.
\end{equation}
Recall that $e^x\geq1+x$ for each real $x$. Then we get $2^{\frac{t+1}{k-t-1}}>1+\frac{t+1}{k-t-1}\cdot\ln 2$ by substituting $x$ to  $\frac{t+1}{k-t-1}\cdot\ln 2$. This together with (\ref{equlemmaspn2lkt}) shows  $2^{\frac{n-(t+s)-1}{k-(t+s)-1}}-1>(t+1)^2(k-t+1)^2$. Thus the desired inequality follows from Lemma \ref{lemmainequntkt}.
\end{proof}
We have introduced in Constructions  \ref{familya} and \ref{familyh} two structures of maximal non-trivial $t$-intersecting families. A simple counting argument using (\ref{equntkt}) gives
\begin{align}
|\mathcal{A}(n,k,t)|&=(t+2)\spn{n-t-1}{k-t-1}-(t+1)\spn{n-t-2}{k-t-2}.\label{equsizea}
\end{align}
For $\mathcal{H}(n,k,t)$, we get from inclusion-exclusion that 
\begin{equation}\label{equsizeh}
|\mathcal{H}(n,k,t)|=\sum_{j=1}^{k-t}(-1)^{j-1}\binom{k-t}{j}\spn{n-t-j}{k-t-j}+t.
\end{equation}
Let us introduce a function which is used frequently:
\begin{equation}\label{equfunr}
	r(n,k,t):=|\mathcal{H}(n,k,t)|-\spn{n-t-1}{k-t-1}+\spn{n-t-2}{k-t-1}.
\end{equation}
Note that (\ref{equnkn-1k}) gives  $r(n,k,t)<|\mathcal{H}(n,k,t)|$.
\begin{lemma}\label{equsizerlb}
Let $k\geq t+3$ and $n\geq2L(k,t)$. Then the following hold.
\begin{itemize}
\item[\rm(i)]$r(n,k,t)>\left(k-t-1-\frac{1}{(t+1)^2}\right)\spn{n-t-1}{k-t-1}+t$.
\item[\rm(ii)]$r(n,t+3,t)>\frac{9}{4}\spn{n-t-1}{2}+t.$
\end{itemize}
\end{lemma}
\begin{proof}
{\rm(i)}\;From (\ref{equsizeh}) and Lemma \ref{lemmain-ex}, we obtain
\begin{align}
	|\mathcal{H}(n,k,t)|&\geq(k-t)\spn{n-t-1}{k-t-1}-\binom{k-t}{2}\spn{n-t-2}{k-t-2}+t.\label{equsizehlb}
\end{align}
Then (i) follows directly from Lemma \ref{lemmaspn2lkt} and (\ref{equfunr}).

{\rm(ii)}\;Using (\ref{equrecurrence}), we derive that  $r(n,t+3,t)=\frac{5}{2}\spn{n-t-1}{2}+t-\frac{7}{2}$. Note that $\spn{n-t-1}{2}\geq2^{2L(k,t)-t-2}-1>14$, then (ii) is true.
\end{proof}
\begin{lemma}\label{lemmat+2small}
	Let $k\geq t+3$ and $n\geq2L(k,t)$. If $\mathcal{F}\subseteq\spn{[n]}{k}$ is a $t$-intersecting family with $\tau_t(\mathcal{F})\geq t+2$, then $|\mathcal{F}|<r(n,k,t)$.
\end{lemma}
\begin{proof}
By (\ref{equthm1}) and Lemma \ref{lemmamono} (i), we have
\begin{equation}\label{equlemmat+2small1}
	|\mathcal{F}|\leq\max\{f(\tau_t(\mathcal{F}),k,t,n),\;f(k,k,t,n)\}\leq\max\{f(t+2,k,t,n),f(k,k,t,n)\}.
\end{equation}
Suppose $k\geq t+4$, then we obtain $|\mathcal{F}|< r(n,k,t)$ by applying Lemma \ref{lemmalemmat+2small}.

It remains to consider the case of $k=t+3$, in which (\ref{equlemmat+2small1}) and (\ref{equfunf}) give
\begin{align*}
|\mathcal{F}|\leq\max\{f(t+2,t+3,t,n),f(t+3,t+3,t,n)\}=f(t+3,t+3,t,n)=64\binom{t+3}{3}.
\end{align*}
Note that $n\geq L(t+3,3)=t+9+4\log_2(t+1)$. Then Lemma \ref{equsizerlb} (ii) yields
\begin{align*}
	r(n,t+3,t)&>\spn{n-t-1}{2}+1=2^{n-t-2}>2^7(t+1)^{4}>64\binom{t+3}{3}.
\end{align*}
Thus $|\mathcal{F}|<r(n,k,t)$ for all $k\geq t+3$.
\end{proof}
This lemma suggests that, to find non-trivial $t$-intersecting families with maximum size for large $n$, we can just consider those with $t$-covering number $t+1$. Let $\mathcal{F}\subseteq\spn{[n]}{k}$ be a $t$-intersecting family with  $\tau_t(\mathcal{F})=t+1$, we use $\mathcal{T}=\mathcal{T}(\mathcal{F})$ to denote the family of $t$-covers of $\mathcal{F}$ with $t+1$ blocks, that is,
\begin{equation}\label{equfamcoverset}
	\mathcal{T}:=\{T:\;T\;\mbox{is a $(t+1)$-partition},\;|T
	\cap F|\geq t\;\mbox{for all}\;F\in\mathcal{F}\}.
\end{equation}
We may regard $\mathcal{T}$ as a family of $(t+1)$-subsets of $2^{[n]}\setminus\{\emptyset\}$, with an additional condition that each member consists of pairwise disjoint sets. 
\begin{lemma}\label{lemmanotintsmall}
Let $k\geq t+3$ and $n\geq2L(k,t)$, and let $\mathcal{F}\subseteq\spn{[n]}{k}$ be a maximal $t$-intersecting family with $\tau_t(\mathcal{F})=t+1$. If there are $T_1,T_2\in\mathcal{T}$ with $|T_1\cap T_2|<t$, then $|\mathcal{F}|<r(n,k,t)$.
\end{lemma}
\begin{proof}
Write $s=|T_1\cap T_2|$ and $w=|\cup(T_1\cup T_2)|$. Let us begin with the following claim.\\
{\bf Claim 1.}\;$w\geq n-2(k-t-2)$.
	
We prove this claim by contradiction. Assume that $w\leq  n-2(k-t-2)-1$, then we can find a $3$-partition of $\overline{\cup(T_1\cup T_2)}$, say $\{A,B,C\}$, such that $|A|=|B|=k-t-2$. Consider 
	\begin{align*}
		F_1&:=T_1\cup[A]\cup\{B\cup C\cup(\cup T_2\setminus\cup T_1)\},\\
		F_2&:=T_2\cup[B]\cup\{A\cup C\cup(\cup T_1\setminus\cup T_2)\}.
	\end{align*}
	Then $F_i\in\spn{[n]}{k}$, $T_i\subseteq F_i$ for $i=1,2$, and $T_1\cap T_2=F_1\cap F_2$. Since $\mathcal{F}$ is maximal, $F_1,F_2\in\mathcal{F}$ and so $t-1\geq|T_1\cap T_2|=|F_1\cap F_2|\geq t$, a contradiction. Hence this claim is true.
	
Since $\mathcal{F}$ is maximal, and $T_1$ and $T_2$ are $t$-covers, every $k$-partition of $[n]$ containing $T_1$ or $T_2$ lies in $\mathcal{F}$. Let $F\in\mathcal{F}\setminus(\mathcal{F}_{T_1}\cup\mathcal{F}_{T_2})$, then $|T_1\cap F|=|T_2\cap F|=t$. If $T_1\cap T_2\nsubseteq F$, then there exists $B\in T_1\cap T_2$ such that $(T_1\cup T_2)\setminus\{B\}\subseteq F$. If $T_1\cap T_2\subseteq F$, then there exist $A_1\in T_1\setminus T_2$ and $A_2\in T_2\setminus T_1$ such that $(T_1\cup T_2)\setminus\{A_1,A_2\}\subseteq F$. Therefore, we can estimate $|\mathcal{F}|$ in two parts, which are
\begin{align}
	\mathcal{F}_1:=\mathcal{F}_{T_1}\cup\mathcal{F}_{T_2}\cup(\cup_{B}\mathcal{F}_{T_1\cup T_2\setminus\{B\}})\;\mbox{and}\;\mathcal{F}_2:=\cup_{(A_1,A_2)}\mathcal{F}_{T_1\cup T_2\setminus\{A_1,A_2\}},\label{equlemmanotintsmallf1f2}
\end{align} 	
where $B$ ranges over the blocks of $T_1\cap T_2$, and $(A_1,A_2)$ ranges over the set $$\mathcal{D}:=\{(A_1,A_2):A_1\in T_1\setminus T_2,A_2\in T_2\setminus T_1\}.$$ 
Using Lemma \ref{equsizerlb} we get $r(n,k,t)>\frac{9}{4}\spn{n-t-1}{k-t-1}$ for $k\geq t+3$. Hence, to prove this lemma, it suffices to verify
\begin{equation}\label{lemmanotintsmallgoal}
	|\mathcal{F}_1|+|\mathcal{F}_2|\leq\frac{9}{4}\spn{n-t-1}{k-t-1}.
\end{equation}
\noindent {\bf Case 1.}\;$s\leq t-2$. 

Note that $|T_1\cup T_2|=2t+2-s\geq t+4$. From (\ref{equlemmanotintsmallf1f2}), (\ref{equntkt}) and Lemma \ref{lemmaspn2lkt}, we obtain
\begin{align*}
	|\mathcal{F}|&\leq2\spn{n-t-1}{k-t-1}+s\spn{n-t-3}{k-t-3}+(t+1-s)^2\spn{n-t-2}{k-t-2}\\
	&<\left(2+\frac{s+(t+1-s)^2}{(t+1)^2(k-t+1)^2}\right)\spn{n-t-1}{k-t-1}\leq\frac{33}{16}\spn{n-t-1}{k-t-1}.
\end{align*}
Then (\ref{lemmanotintsmallgoal}) holds.\\
\noindent {\bf Case 2.}\;$s=t-1$.

Using Claim 1 we infer
\begin{align}
	w&\geq2L(k,t)-2(k-t-2)\nonumber\\
	&=2t+8+2(k-t+1)(-1+\log_2(t+1)(k-t+1))\nonumber\\
	&\geq4k-2t+12\geq2t+24.\label{equlemmanotintsmall1}
\end{align}
The following claim is deduced by the maximality of $\mathcal{F}$ as well. \\
\noindent {\bf Claim 2.}\;$|\cup T_i|\geq t+6$ for $i=1,2.$

We prove by contradiction again. Without loss of generality, assume that $|\cup T_1|\leq t+5$. Suppose that $T_2\setminus T_1=\{X,X'\}$ and $|X|\geq|X'|$. It follows from $X\cup X'\supseteq\cup(T_1\cup T_2)\setminus\cup T_1$ and (\ref{equlemmanotintsmall1}) that 
$$|X\cup X'|\geq w-(t+5)\geq 4k-3t+7\geq2(k-t+5),$$
 and so $|X|\geq k-t+5$. Then we can pick $Y\in\binom{X}{k-t-2}$ with $Y\cap(\cup(T_1\setminus T_2))=\emptyset$ as $|\cup(T_1\setminus T_2)|\leq(t+5)-(t-1)=6$. Then 
$$F_3:=T_1\cup[Y]\cup\{\overline{(\cup T_1)\cup Y}\}$$
is a $k$-partition of $[n]$, and $X\notin F_3$. Since $T_1\subseteq F_3$ and $\mathcal{F}$ is maximal, $F_3\in\mathcal{F}$. Then $X'\in F_3$ as $T_2$ is a $t$-cover of $\mathcal{F}$. It is clearly that $X'\notin T_1\cup[Y]$, implying that $X'=\overline{(\cup T_1)\cup Y}$. Now $F_3=T_1\cup[Y]\cup\{X'\}$, and then every element of $X\setminus Y$ must appear in some block of $T_1\setminus T_2$, namely, $X\setminus Y\subseteq\cup(T_1\setminus T_2)$. However, this leads to $6\geq|\cup(T_1\setminus T_2)|\geq|X|-|Y|\geq7$, which is impossible. Thus Claim 2 holds.

From Claim 2, (\ref{equnkn-1k}) and Lemma \ref{lemmaspn2lkt}, we derive that
\begin{align}
	|\mathcal{F}_1|&<2\spn{n-t-6}{k-t-1}+(t-1)\spn{n-t-2}{k-t-2}\nonumber\\
	&<\left(2(k-t-1)^{-5}+(k-t+1)^{-2}\right)\spn{n-t-1}{k-t-1}\leq\frac{1}{8}\spn{n-t-1}{k-t-1}.\label{equlemmanotintsmall3}
\end{align}
Therefore, to prove (\ref{lemmanotintsmallgoal}), it remains to prove $|\mathcal{F}_2|\leq\frac{17}{8}\spn{n-t-1}{k-t-1}$.\\
{\bf Case 2.1.}\;$T_1\cap T_2$ contains a block of size at least two. 

Now from (\ref{equnkn-1k}) and (\ref{equlemmanotintsmallf1f2}), we get
\begin{align*}
	|\mathcal{F}_2|\leq4\spn{n-t-2}{k-t-1}<4(k-t-1)^{-1}\spn{n-t-1}{k-t-1}\leq2\spn{n-t-1}{k-t-1}.
\end{align*}
\noindent {\bf Case 2.2.}\;$T_1\cap T_2$ consists of singletons. 

By Claim 2, both $T_1\setminus T_2$ and $T_2\setminus T_1$ contain some block of size larger than one. Denote $$\mathcal{D}':=\{(A_1,A_2)\in\mathcal{D}:|C|=1\;\mbox{for all}\;C\in T_1\cup T_2\setminus\{A_1,A_2\}\},$$
then $|\mathcal{D}'|\leq1$. If $\mathcal{D}'=\emptyset$ then $|\cup(T_1\cup T_2\setminus\{A_1,A_2\})|\geq t+2$ for all $(A_1,A_2)\in\mathcal{D}$. Otherwise, assume that $\mathcal{D}'=\{(A,B)\}$. Then  $|A|\geq2,|B|\geq2$, and (\ref{equlemmanotintsmall1}) gives $|A\cup B|\geq t+23$. Hence, we deduce from (\ref{equnkn-1k}) that 
\begin{align*}
	|\mathcal{F}_2|&\leq\max\left\{4\spn{n-t-2}{k-t-1},\spn{n-t-1}{k-t-1}+2\spn{n-t-2}{k-t-1}+\spn{n-2t-22}{k-t-1}\right\}\\
	&<(1+2(k-t-1)^{-1}+(k-t-1)^{-22})\spn{n-t-1}{k-t-1}<2.01\spn{n-t-1}{k-t-1}.
\end{align*}

Thus (\ref{lemmanotintsmallgoal}) is true, which completes the proof.
\end{proof}
We are in a position to investigate families satisfying the following assumption.
\begin{assumption}\label{assumption}
	Let $k\geq t+3$ and $n\geq2L(k,t)$, where $L(k,t)$ is defined in (\ref{equfunlkt}). Let $\mathcal{F}\subseteq\spn{[n]}{k}$ be a maximal $t$-intersecting family with $t$-covering number $t+1$. Let $\mathcal{T}$ be as introduced in (\ref{equfamcoverset}). Suppose that $\mathcal{T}$ is $t$-intersecting as a set system.
\end{assumption}
\begin{lemma}\label{lemmatnontrivial}
	Let $n,k,t,\mathcal{F}$ and $\mathcal{T}$ be as in Assumption \ref{assumption}. If $|\cap\mathcal{T}|<t$, then there is a $(t+2)$-partition $Z$ such that $\mathcal{F}=\mathcal{A}(Z)$, and we have $|\mathcal{F}|\leq|\mathcal{A}(n,k,t)|$, with equality only if $\mathcal{F}\cong\mathcal{A}(n,k,t)$.
\end{lemma}
\begin{proof}
	Obviously $|\mathcal{T}|\geq2$. Let $T_1, T_2$ be two distinct elements of $\mathcal{T}$. Then $|T_1\cap T_2|=t$.  Since $|\cap\mathcal{T}|<t$, there exists $T_3\in\mathcal{T}$ such that $|T_3\cap(T_1\cap T_2)|<t$. Then $|T_3\cap T_1|=|T_3\cap T_2|=t$ yields that $T_3=(T_1\cup T_2)\setminus\{B\}$ for some $B\in T_1\cap T_2$. In particular, this implies that the block of $T_1\setminus T_2$ and that of $T_2\setminus T_1$ are disjoint as they lie in the partition $T_3$, and thus $T_1\cup T_2$ is a $(t+2)$-partition. Every member of $\mathcal{T}$ or $\mathcal{F}$ shares at least $t$ blocks with $T_i$ whenever $i\in\{1,2,3\}$, from which we obtain that  $\cup\mathcal{T}=T_1\cup T_2\cup T_3=T_1\cup T_2$, and then
\begin{equation*}
	\mathcal{F}\subseteq\left\{F\in\spn{[n]}{k}:|F\cap(T_1\cup T_2)|\geq t+1\right\}=\mathcal{A}(T_1\cup T_2).
\end{equation*}
By the maximality of $\mathcal{F}$, we obtain $\mathcal{F}=\mathcal{A}(T_1\cup T_2)$. 

Write $Z=T_1\cup T_2$ for short. If $Z$ consists of $t+2$ singletons, then $\mathcal{F}\cong\mathcal{A}(n,k,t)$. Suppose that $Z$ contains some block, say $B$, with $|B|\geq2$. Then $|\cup(Z\setminus\{C\})|\geq t+2$ for all $C\in Z\setminus\{B\}$, and so
\begin{align*}
	|\mathcal{F}|&<\spn{n-|\cup(Z\setminus\{B\})|}{k-t-1}+\sum_{C\in Z\setminus\{B\}}\spn{n-|\cup(Z\setminus\{C\})|}{k-t-1}\\
	&\leq\spn{n-t-1}{k-t-1}+(t+1)\spn{n-t-2}{k-t-1}.
\end{align*}
This together with (\ref{equsizea}) and (\ref{equrecurrence}) yields
\begin{align*}
|\mathcal{F}|-|\mathcal{A}(n,k,t)|&<-(t+1)\left(\spn{n-t-1}{k-t-1}-\spn{n-t-2}{k-t-1}-\spn{n-t-2}{k-t-2}\right)\\
&=-(t+1)(k-t-2)\spn{n-t-2}{k-t-1}<0.
\end{align*}
This finishes the proof.
\end{proof}
\begin{lemma}\label{lemmattrivial1}
Let $n,k,t,\mathcal{F}$ and $\mathcal{T}$ be as in Assumption \ref{assumption}. If $|\mathcal{T}|=1$, then $|\mathcal{F}|<r(n,k,t)$.
\end{lemma}
\begin{proof}
Suppose $\mathcal{T}=\{T\}$. Since $T$ is a $t$-cover, every member of $\mathcal{F}$ contains all but at most one block of $T$, and hence $$\mathcal{F}\setminus\mathcal{F}_T=\textstyle\cup_{X\in\binom{T}{t}}(\mathcal{F}_X\setminus\mathcal{F}_T).$$ 
Fix an $X\in\binom{T}{t}$. Since $\mathcal{F}$ is non-trivial, there exists $G\in\mathcal{F}$ with $X\nsubseteq G$. Then $G\cap T=T\setminus\{B\}$ for some $B\in X$, and so $|F\cap(G\cup\{B\})|\geq t+1$ for each $F\in\mathcal{F}_{X}\setminus\mathcal{F}_T$. It follows that
 \begin{equation}\label{equlemmattrivial13}
 \mathcal{F}_{X}\setminus\mathcal{F}_T=\cup_{H\in\mathcal{H}}\mathcal{F}_H,
 \end{equation}
  where $\mathcal{H}$ is the collection of $(t+1)$-subsets $H$ of $G\cup\{B\}$ satisfying  $X\subseteq H$ and $H\neq T$. Clearly, $|\mathcal{H}|=k-t$. Let $H\in\mathcal{H}$. If $H$ is not a partition, then obviously $\mathcal{F}_H=\emptyset$. Suppose that $H$ is a partition. Note that $H$ does not lies in $\mathcal{T}=\{T\}$ and so it is not a $t$-cover of $\mathcal{F}$. Hence $|H\cap F'|<t$ for some $F'\in\mathcal{F}$. Then $|H\cap F'|\geq|X\cap F'|\geq t-1$ gives $|H\cap F'|=t-1$. By applying Lemma \ref{lemmainductive} to $\ell=k,s=t+1$ and $r=t-1$, we obtain
 $|\mathcal{F}_H|\leq(k-t+1)\spn{n-t-2}{k-t-2}$. 
 By combining this with (\ref{equlemmattrivial13}), we deduce from (\ref{equntkt}) and  Lemma \ref{lemmaspn2lkt} that
 \begin{align*}
 |\mathcal{F}|&\leq|\mathcal{F}_T|+\textstyle\sum_{X\in\binom{T}{t}}|\mathcal{F}_X\setminus\mathcal{F}_T|\\
 &\leq\spn{n-t-1}{k-t-1}+(t+1)(k-t)(k-t+1)\spn{n-t-2}{k-t-2}<2\spn{n-t-1}{k-t-1}. \end{align*}
 Finally, we get $|\mathcal{F}|<r(n,k,t)$ from Lemma \ref{equsizerlb}.
\end{proof}
\begin{lemma}\label{lemmattrivial2}
Let $n,k,t,\mathcal{F}$ and $\mathcal{T}$ be as in Assumption \ref{assumption}. Suppose that  $|\cap\mathcal{T}|=t$. Then $T\cap T'=\cap\mathcal{T}$ for all $T\neq T'\in\mathcal{T}$, and $\cup\mathcal{T}$ is a partition with at most $k$ blocks. Moreover, if $F\in\mathcal{F}$ does not contain $\cap\mathcal{T}$, then $(\cup\mathcal{T})\cap F=(\cup\mathcal{T})\setminus\{B\}$ for some $B\in\cap\mathcal{T}$.
\end{lemma}
\begin{proof}
Since $|\cap\mathcal{T}|=t$ and $\mathcal{T}$ consists of $(t+1)$-sets, we have  $|\mathcal{T}|\geq2$, and every two members of $\mathcal{T}$ have intersection $\cap\mathcal{T}$. Set $T=(\cap\mathcal{T})\cup\{B_T\}$ for $T\in\mathcal{T}$, then
\begin{equation}\label{equlemmattrivial21}
\cup\mathcal{T}=(\cap\mathcal{T})\cup\{B_T:T\in\mathcal{T}\}.
\end{equation}
 Note that $\mathcal{F}$ is non-trivial, some partitions in $\mathcal{F}$ do not contain $\cap\mathcal{T}$. Fix any one of them, say $F_0$. Then for every $T\in\mathcal{T}$, we have from $|T\cap F_0|\geq t$ that $B_T\in F_0$, and thus $\{B_T:T\in\mathcal{T}\}\subseteq F_0$. This yields that $\{B_T:T\in\mathcal{T}\}$ is a partition, and so is $\cup\mathcal{T}$. From (\ref{equlemmattrivial21}), we obtain that $(\cup\mathcal{T})\cap F_0=(\cup\mathcal{T})\setminus\{B\}$ for some $B\in\cap\mathcal{T}$. It follows from $F_0\neq\cup\mathcal{T}$ that $|(\cup\mathcal{T})\cap F_0|\leq k-1$, and so $|\cup\mathcal{T}|\leq(k-1)+1=k$.
\end{proof}
\begin{lemma}\label{lemmattrivial3baby}
Let $n,k,t,\mathcal{F}$ and $\mathcal{T}$ be as in Assumption \ref{assumption}. Suppose that $|\cap\mathcal{T}|=t$ and $t+2\leq|\cup\mathcal{T}|\leq k-2$. Then $|\mathcal{F}|<r(n,k,t)$.
\end{lemma}
\begin{proof}
Write $X:=\cap\mathcal{T}$, $M:=\cup\mathcal{T}$ and $m:=|M|$ for short. Since $\mathcal{F}$ is non-trivial,  $\mathcal{F}\setminus\mathcal{F}_X\neq\emptyset$. We will estimate the size of $\cup_{T\in\mathcal{T}}\mathcal{F}_T$, $\mathcal{F}_X\setminus\cup_{T\in\mathcal{T}}\mathcal{F}_T$ and $\mathcal{F}\setminus\mathcal{F}_X$, respectively. 

Firstly, we have
\begin{equation}\label{equlemmattrivial3case31}
	|\cup_{T\in\mathcal{T}}\mathcal{F}_T|\leq\sum_{T\in\mathcal{T}}\spn{n-|\cup T|}{k-t-1}\leq(m-t)\spn{n-t-1}{k-t-1}.
\end{equation}

To proceed, fix an $F\in\mathcal{F}\setminus\mathcal{F}_X$. From Lemma \ref{lemmattrivial2}, we have $F\cap M=M\setminus\{B\}$ for some $B\in X$. Then $F\setminus M$ has $k-m+1$ blocks, and at least one of them intersects $B$ as $B\subseteq\cup(F\setminus M)$. Let $H$ be the collection of blocks of $F\setminus M$ which is disjoint from $B$, then $|H|\leq k-m$. For each  $G\in\mathcal{F}_X\setminus\cup_{T\in\mathcal{T}}\mathcal{F}_T$, we have $G\cap M=X$, and then $\emptyset\neq G\cap(F\setminus M)\subseteq H$ as $G\cap F\cap M=X\setminus\{B\}$ has only $t-1$ blocks and $B\in G$. Therefore, 
\begin{equation}\label{equlemmattrivial3part2}
	\mathcal{F}_X\setminus\cup_{T\in\mathcal{T}}\mathcal{F}_T=\cup_{C\in H}\mathcal{F}_{X\cup\{C\}}.
\end{equation}
Let $C\in H$. Note that $C\notin M$, then  $X\cup\{C\}$ is not a $t$-cover of $\mathcal{F}$. Also note that $X=\cap\mathcal{T}$. Then some member of $\mathcal{F}$ shares exactly $t-1$ blocks with $X\cup\{C\}$. Combining this with Lemma \ref{lemmainductive} and (\ref{equlemmattrivial3part2}), we derive that
\begin{equation}\label{equlemmattrivial3case33}
	|\mathcal{F}_X\setminus\cup_{T\in\mathcal{T}}\mathcal{F}_T|\leq\sum_{C\in H}|\mathcal{F}_{X\cup\{C\}}|\leq(k-m)(k-t+1)\spn{n-t-2}{k-t-2}.
\end{equation}

Now consider $\mathcal{F}\setminus\mathcal{F}_X$. By Lemma \ref{lemmattrivial2}, we have $\mathcal{F}\setminus\mathcal{F}_X=\cup_{B\in X}\mathcal{F}_{M\setminus\{B\}}$. Fix a $B\in X$. Suppose firstly that $m=t+2$.  Then $|M\setminus\{B\}|=t+1$, and $B\in X=\cap\mathcal{T}$ yields that  $M\setminus\{B\}$ is not a $t$-cover. Hence, $|F'\cap(M\setminus\{B\})|=t-1$ for some $F'\in\mathcal{F}$. It follows from Lemma \ref{lemmainductive} that $|\mathcal{F}_{M\setminus\{B\}}|\leq(k-t+1)\spn{n-t-2}{k-t-2}$. If $m\geq t+3$, then $|M\setminus\{B\}|\geq t+2$, and so  $|\mathcal{F}_{M\setminus\{B\}}|\leq\spn{n-t-2}{k-t-2}$. Therefore, we obtain
\begin{align}
	|\mathcal{F}\setminus\mathcal{F}_X|\leq\sum_{B\in X}|\mathcal{F}_{M\setminus\{B\}}|\leq t(k-t+1)\spn{n-t-2}{k-t-2}.\label{equlemmattrivial3part32}
\end{align} 

Finally, from (\ref{equlemmattrivial3case31}), (\ref{equlemmattrivial3case33}), (\ref{equlemmattrivial3part32}) and $t+2\leq m\leq k-2$, we obtain 
$$|\mathcal{F}|\leq(k-t-2)\spn{n-t-1}{k-t-1}+(k-2)(k-t+1)\spn{n-t-2}{k-t-2}.$$
It is routine to check that $(k-2)<(t+1)(k-t+1)$ for $k\geq t+3$. Then Lemmas \ref{lemmaspn2lkt} and  \ref{equsizerlb} (i) yields $|\mathcal{F}|<\left(k-t-2+\frac{1}{t+1}\right)\spn{n-t-1}{k-t-1}<r(n,k,t)$, as desired.
\end{proof}
\begin{lemma}\label{lemmattrivial3}
	Let $n,k,t,\mathcal{F}$ and $\mathcal{T}$ be as in Assumption \ref{assumption}. Suppose that $|\cap\mathcal{T}|=t$ and $|\cup\mathcal{T}|\in\{k-1,k\}$. The following hold.
\begin{itemize}
\item[\rm (i)]If $\mathcal{F}$ is isomorphic to neither $\mathcal{H}(n,k,t)$ nor $\mathcal{H}_1(n,k,t)$, then $|\mathcal{F}|<r(n,k,t)$.
\item[\rm (ii)]We have $|\mathcal{F}|\leq|\mathcal{H}(n,k,t)|$, with equality only if $\mathcal{F}\cong\mathcal{H}(n,k,t)$.
\end{itemize}
\end{lemma}
\begin{proof}
 Note that Lemma \ref{lemmahh1v} (i) gives $\max\{r(n,k,t),|\mathcal{H}_1(n,k,t)|\}<|\mathcal{H}(n,k,t)|$, then (ii) follows directly from (i), and so we need only to prove (i). Let us introduce two expressions.
 \begin{align}
 	u_1(n,k,t)&:=(k-t-2)\spn{n-t-1}{k-t-1}+2\spn{n-t-2}{k-t-1}+2t.\label{equfunu1}\\
 	u_2(n,k,t)&:=\sum_{j=1}^{3}(-1)^{j-1}\binom{k-t-1}{j}\spn{n-t-j}{k-t-j}+\spn{n-t-3}{k-t-1}+2t.\label{equfunu2}
 \end{align} 
We will prove in Lemma \ref{lemmathm34} that $\max\{u_1(n,k,t),u_2(n,k,t)\}<r(n,k,t)$. Then it suffices to prove that, if $\mathcal{F}$ is isomorphic to neither $\mathcal{H}(n,k,t)$ nor $\mathcal{H}_1(n,k,t)$, then 
 \begin{equation}\label{equlemmattrivial3goal}
 	|\mathcal{F}|\leq\max\{u_1(n,k,t),u_2(n,k,t)\}.
 \end{equation}

 For simplicity, for partitions $Y$ and $N$ with $Y\subsetneqq N$, set
\begin{equation*}
	\mathcal{G}(Y,N):=\left\{G\in\spn{[n]}{k}:Y\subseteq G,\;G\cap(N\setminus Y)\neq\emptyset\right\}.
\end{equation*}
Set $X:=\cap\mathcal{T}$ and  $M:=\cup\mathcal{T}$. By Lemma \ref{lemmattrivial2} and the maximality of $\mathcal{F}$, we have 
\begin{equation}\label{equgxm'}
\mathcal{G}(X,M)=\left\{G\in\spn{[n]}{k}:T\subseteq G\;\mbox{for some}\;T\in\mathcal{T}\right\}\subseteq\mathcal{F}_X.
\end{equation}
\noindent {\bf Case 1.}\;$|M|=k$. 

In this case, Lemma \ref{lemmattrivial2} gives
\begin{equation}\label{equlemmattrivial31}
	\mathcal{F}\setminus\mathcal{F}_X\subseteq\{(M\setminus\{B\})\cup\{B\cup\left(\overline{\cup M}\right)\}:B\in X\}.
\end{equation} 
Note that $\mathcal{F}\setminus\mathcal{F}_X\neq\emptyset$ implies $\cup M\subsetneqq[n]$. Fix an $F\in\mathcal{F}\setminus\mathcal{F}_X$, and let $G\in\mathcal{F}_X$. Then the block of $F\setminus M$ does not lie in $G$ as  $X\subseteq G$. It follows that $F\cap G\subseteq M$, and then $G\cap(M\setminus X)\neq\emptyset$ as $|F\cap G|\geq t$ and $X\nsubseteq F$. Hence 
$\mathcal{F}_X\subseteq\mathcal{G}(X,M)$. Now (\ref{equgxm'}) yields   $\mathcal{F}_X=\mathcal{G}(X,M)$. Since  $\mathcal{F}$ is maximal, equality in (\ref{equlemmattrivial31}) holds, and consequently $\mathcal{F}=\mathcal{H}(X,M).$

If $|\cup X|>t$, then $|\cup X|+|C|\geq t+2$ for each $C\in M\setminus X$, and so  $|\mathcal{F}|\leq(k-t)\spn{n-t-2}{k-t-1}+t<u_1(n,k,t)$. If $X$ contains only singletons, and $M\setminus X$ contains at least two blocks of size larger than one, then $|\mathcal{F}|\leq(k-t-2)\spn{n-t-1}{k-t-1}+2\spn{n-t-2}{k-t-1}+t<u_1(n,k,t)$. Suppose that there exists $B_0\in M\setminus X$ such that $M\setminus\{B_0\}$ consists of  singletons. Now $\mathcal{F}\cong\mathcal{H}(n,k,t)$ when $|B_0|=1$, and $\mathcal{F}\cong\mathcal{H}_1(n,k,t)$ when $|B_0|=2$. If $|B_0|\geq3$, then we   rewrite $\mathcal{F}_X$ as 
$$\mathcal{G}(X,M\setminus\{B_0\})\cup\left\{F\in\spn{[n]}{k}:X\cup\{B_0\}\subseteq F\right\},$$
and then derive from Lemma \ref{lemmain-ex} and (\ref{equntkt}) that
\begin{align*}
	|\mathcal{F}|&\leq\sum_{j=1}^{3}(-1)^{j-1}\binom{k-t-1}{j}\spn{n-t-j}{k-t-j}+\spn{n-t-|B_0|}{k-t-1}+t<u_2(n,k,t).
\end{align*} 
Therefore, (\ref{equlemmattrivial3goal}) holds if $\mathcal{F}$ is isomorphic to neither $\mathcal{H}(n,k,t)$ nor $\mathcal{H}_1(n,k,t)$.\\
\noindent {\bf Case 2.}\;$|M|=k-1$.

If $M$ contains some block of size at least two, then there exists $C\in M\setminus X$ such that $|\cup X|+|C|\geq t+2$, and then
\begin{equation}\label{equlemmattrivial3case21}
	|\mathcal{G}(X,M)|\leq\sum_{C\in M\setminus X}\spn{n-(|\cup X|+|C|)}{k-(t+1)}\leq u_1(n,k,t)-\spn{n-t-2}{k-t-1}-2t.
\end{equation}
If $M$ consists of singletons, then we use (\ref{equgxm'}) and Lemma \ref{lemmain-ex} to establish
\begin{align}
	|\mathcal{G}(X,M)|&\leq u_2(n,k,t)-\spn{n-t-3}{k-t-1}-2t.\label{equlemmattrivial3case22}
\end{align}
From (\ref{equnkn-1k}) and Lemma \ref{lemmaspn2lkt'} (ii), we have   $\spn{n-t-2}{k-t-1}>\spn{n-t-3}{k-t-1}>t\spn{n-t-2}{k-t-2}$. By combining this with  (\ref{equlemmattrivial3case21}) and (\ref{equlemmattrivial3case22}), to prove   (\ref{equlemmattrivial3goal}), it suffices to prove
\begin{equation}\label{equlemmattrivial3goal'}
|\mathcal{F}\setminus\mathcal{G}(X,M)|\leq t\spn{n-t-2}{k-t-2}+2t.
\end{equation}
Let us consider whether $\mathcal{F}_X=\mathcal{G}(X,M)$.\\
{\bf Case 2.1.}\;$\mathcal{F}_X=\mathcal{G}(X,M)$.

From Lemma \ref{lemmattrivial2} and the maximality of $\mathcal{F}$, we deduce that
\begin{align*}
	\mathcal{F}&=\mathcal{G}(X,M)\cup\left\{(M\setminus\{B\})\cup P:B\in X,P\in\spn{B\cup\left(\overline{\cup M}\right)}{2}, B\notin P\right\}.
\end{align*}
Let $B\in X$, then $M\setminus\{B\}$ forms a $t$-covers of $\mathcal{F}$ with $k-2$ blocks. Note that $M\setminus\{B\}\notin\mathcal{T}$ as $X=\cap\mathcal{T}$, from which we infer $k-2\neq t+1$, or equivalently, $k\geq t+4$. It follows that
 \begin{align}
 	|\mathcal{F}\setminus\mathcal{G}(X,M)|&\leq\sum_{B\in X}\left(\spn{|B|+n-|\cup M|}{2}-1\right)\nonumber\\&<t\spn{n-k+2}{2}
 	=t\spn{n-(k-2)}{k-(k-2)}\leq t\spn{n-t-2}{k-t-2},\nonumber
 \end{align}
which justifies (\ref{equlemmattrivial3goal'}).

\noindent {\bf Case 2.2.}\;$\mathcal{F}_X\setminus\mathcal{G}(X,M)\neq\emptyset$.

Let  $G\in\mathcal{F}_X\setminus\mathcal{G}(X,M)$ and let $F\in\mathcal{F}\setminus\mathcal{F}_X$. Then $M\cap G=X$. From Lemma \ref{lemmattrivial2}, we have  $M\setminus F=\{B\}$ for some $B\in X$, and  $F\setminus M$ forms a $2$-partition of $B\cup\left(\overline{\cup M}\right)$. Since $|F\cap G|\geq t$ and $F\cap G\cap M=X\setminus\{B\}$ has only $t-1$ blocks, $G$ contains some block of $F\setminus M$. This block must be disjoint from $B$ as $B\in G$, namely, it is a subset of $\overline{\cup M}$. Set
\begin{equation*}
	Q:=\{C\subseteq\overline{\cup M}:C\in F_1\setminus M\;\mbox{for some}\;F_1\in\mathcal{F}\setminus\mathcal{F}_X\}.
\end{equation*} 
Then one can check that 
\begin{equation*}
	\mathcal{F}_X\setminus\mathcal{G}(X,M)\subseteq\mathcal{G}':=\left\{G\in\spn{[n]}{k}: M\cap G=X, X\cup Q\subseteq G\right\}
\end{equation*} 
and
\begin{equation*}
	\mathcal{F}\setminus\mathcal{F}_X\subseteq\mathcal{G}'':=\left\{(M\setminus\{B\})\cup\{C,B\cup\overline{C\cup(\cup M)}\}: B\in X, C\in Q\right\}.
\end{equation*}
Therefore, we deduce from the maximality of $\mathcal{F}$ that $$\mathcal{F}=\mathcal{G}(X,M)\cup\mathcal{G'}\cup\mathcal{G}''.$$

We claim that $2\leq|Q|\leq k-t-1$. Firstly, assume that $|Q|=1$. Then $X\cup Q$ forms a $t$-cover of $\mathcal{F}$ with $t+1$ blocks, and so $X\cup Q\in\mathcal{T}$. However, this implies that $\cup Q\subseteq M$, which is a contradiction. Thus $|Q|\geq2$. For a $k$-partition in  $\mathcal{F}_X\setminus\mathcal{G}(X,M)$, some of its blocks intersects $\cup(M\setminus X)$, and necessarily this block does not lie in $X\cup Q$. Then $k\geq1+|X|+|Q|$, and thus $|Q|\leq k-t-1$, as claimed.

 Using Lemma \ref{lemmaqt}, we derive
\begin{align*}
	|\mathcal{F}\setminus\mathcal{G}(X,M)|=|\mathcal{G}'|+|\mathcal{G}''|&\leq\spn{n-t-|Q|}{k-t-|Q|}+|Q|t\leq\spn{n-t-2}{k-t-2}+2t,
\end{align*}  
and thus (\ref{equlemmattrivial3goal'}) holds. 

This completes the proof.  
\end{proof}
\noindent {\bf Proof of Theorem \ref{thm2}}.\;From Lemmas  \ref{lemmat+2small}--\ref{lemmattrivial1}, \ref{lemmattrivial3baby} and  \ref{lemmattrivial3} (ii), we obtain that   $|\mathcal{F}|\leq\max\{|\mathcal{A}(n,k,t)|,|\mathcal{H}(n,k,t)|\}$, and equality holds implies that $\mathcal{F}\cong\mathcal{A}(n,k,t)$ or $\mathcal{F}\cong\mathcal{H}(n,k,t)$. Then this theorem follows from Lemma \ref{lemmawholargest}.{\hfill$\square$}

\begin{corollary}\label{coro}
Let $k\geq t+3$ and $n\geq2L(k,t)$. If $\mathcal{F}\subseteq\spn{[n]}{k}$ is a maximal $t$-intersecting family with $|\mathcal{F}|<\spn{n-t}{k-t}$, then $|\mathcal{F}|\leq\spn{n-t-1}{k-t}$. Moreover, equality holds only if $\mathcal{F}$ consists of all $k$-partitions of $[n]$ with $t-1$ fixed singletons and one fixed block of size two.
\end{corollary}
\begin{proof}
Suppose firstly that $\mathcal{F}$ is trivial. Then as presented in Construction \ref{familyf}, there is a $t$-partition $X$ such that $\mathcal{F}=\left\{F\in\spn{[n]}{k}:X\subseteq F\right\}$, and $|\mathcal{F}|=\spn{n-|\cup X|}{k-t}$. Then  $|\mathcal{F}|<\spn{n-t}{k-t}$ yields $|\cup X|>t$, and hence $|\mathcal{F}|\leq\spn{n-t-1}{k-t}$, with equality precisely if $|\cup X|=t+1$, namely, $X$ contains $t-1$ singletons and one block of size two.

 It remains to verify $|\mathcal{F}|<\spn{n-t-1}{k-t}$ provided that $\mathcal{F}$ is non-trivial. In fact, in this case we have from Theorem \ref{thm2}, (\ref{equsizea}) and (\ref{equsizeh}) that $|\mathcal{F}|\leq\max\{t+2,k-t\}\cdot\spn{n-t-1}{k-t-1}+t$. On the other hand, Lemma \ref{lemmaspn2lkt'} (i) gives $\spn{n-t-1}{k-t}>(t+1)^2(k-t+1)\spn{n-t-1}{k-t-1}$. Thus $|\mathcal{F}|<\spn{n-t-1}{k-t}$.
\end{proof}

\noindent {\bf Proof of Theorem \ref{thm3}}. Let $\mathcal{F}\subseteq\spn{[n]}{k}$ be a maximal non-trivial $t$-intersecting family satisfying (i) or (ii). Then $|\mathcal{F}|\geq r(n,k,t)$. From Lemmas \ref{lemmat+2small}--\ref{lemmattrivial1}, \ref{lemmattrivial3baby} and \ref{lemmattrivial3} (i), we have  $\mathcal{F}\cong\mathcal{H}(n,k,t)$ or $\mathcal{H}_1(n,k,t)$, or $\mathcal{F}=\mathcal{A}(Z)$ for some $(t+2)$-partition $Z$.

(i)\;Suppose $k\geq2t+3$. We need only to prove that, if $\mathcal{F}=\mathcal{A}(Z)$ for some $(t+2)$-partition $Z$, then $|\mathcal{F}|<r(n,k,t)$. This follows directly from Lemmas \ref{lemmatnontrivial} and \ref{lemmahh1v} (ii).

(ii)\;Suppose $k\leq2t+2$ and $|\mathcal{F}|>|\mathcal{H}(n,k,t)|$. It follows from Lemma \ref{lemmahh1v} (i) that $|\mathcal{F}|>|\mathcal{H}_1(n,k,t)|$, and so there is a $(t+2)$-partition $Z$ such that $\mathcal{F}=\mathcal{A}(Z)$. {\hfill$\square$}
\begin{remark}\label{rmk}
Suppose that $k\sim(1+a)t$ for some $a>0$. Let $c=(\ln2)/2$ and $n=cL(k,t)$. A routine but tedious computation involving (\ref{equspn}) and Lemma \ref{lemmain-ex} derives the estimations $\spn{n-t-1}{k-t}\sim e^{-x_0}\frac{(k-t)^{n-t-1}}{(k-t)!}$ and $\spn{n-t-2}{k-t-1}\sim e^{-x_0}\frac{(k-t-1)^{n-t-2}}{(k-t-1)!}$, where $x_0=ae^{(1-c)a^{-1}}$. Then one can check that $(k-t)\spn{n-t-1}{k-t}/((k-t-1)\spn{n-t-2}{k-t-1})\sim e^{-(1-c)a^{-1}}t$, and from which obtain  that $\spn{n-t}{k-t}-|\mathcal{A}(n,k,t)|=(k-t)\spn{n-t-1}{k-t}-(t+1)(k-t-1)\spn{n-t-2}{k-t-1}<0$ for all sufficiently large $t$. Therefore, the least value for $n$ such that Theorem \ref{thm1} holds equals $\Theta(L(k,t))$.
\end{remark}
\section{Some inequalities}\label{section4}
	In this section, we prove several inequalities used in this paper. We note that $L(k,t)$ and $f(m,k,t,n)$ are defined in (\ref{equfunlkt}) and (\ref{equfunf}), respectively, and the functions $r(n,k,t),u_1(n,k,t)$ and $u_2(n,k,t)$ required in Section \ref{section3} are defined in (\ref{equfunr}), (\ref{equfunu1}) and (\ref{equfunu2}), respectively.
	\begin{lemma}\label{lemmalrtr}
		For $\ell\geq t\geq r\geq1$, we have $\binom{\ell-r}{t-r}\leq(\ell-t+1)^{t-r}.$
	\end{lemma}
	\begin{proof}
		It suffices to suppose $t>r$. Note that $\frac{b}{a}<\frac{b-1}{a-1}$ for $b>a>1$, then 
		$$\binom{\ell-r}{t-r}=\prod_{i=0}^{t-r-1}\frac{\ell-r-i}{t-r-i}\leq(\ell-t+1)^{t-r},$$
		as desired.
	\end{proof}
\begin{lemma}\label{lemmaqkt}
	Let $t\geq1$. Then $Q(s,t):=s^{L(s+t,t)-2s-t+1}\binom{s+t}{t}^{-1}$ is increasing on $s\geq2$. In particular, $Q(s,t)\geq18.$
\end{lemma}
\begin{proof}
	Note that $L(s+t,t)=t+1+(s+1)\cdot\log_2s\cdot\log_s(t+1)(s+1)$, then
	\begin{align*}
	Q(s,t)&=s^{2-2s}((t+1)(s+1))^{(s+1)\log_2s}\binom{s+t}{t}^{-1}\\
	&=\frac{s^2(t+1)^{2s}}{\binom{s+t}{t}}\left(1+\frac{1}{s}\right)^{2s}((t+1)(s+1))^{-2s+(s+1)\log_2s}.
	\end{align*}
	Now it is routine to check that each of the three factors above is increasing on $s\geq2$, then so is $Q(s,t)$. Hence $Q(s,t)\geq Q(2,t)=\frac{27(t+1)^2}{2(t+2)}\geq18.$
\end{proof}
To proceed, we refer to Rennie and Dobson \cite{Rennie-Dobson-1969}, who established useful bounds for the Stirling partition number.
\begin{lemma}[\cite{Rennie-Dobson-1969}]\label{lemmalb}
	For $1\leq r<n$, we have
	\begin{equation*}
		\spn{n}{r}\geq \frac{1}{2}(r^2+r+2)r^{n-r-1}-1.
	\end{equation*}
\end{lemma}
\begin{lemma}\label{lemmamono}
	Let $k\geq t+2$ and $n\geq L(k,t)$. The following hold.
\begin{itemize}
 \item[\rm(i)]\;The functions $f(m,k,t,n)$ and $(k-t+1)^{m-t}\spn{n-m}{k-m}$ are both strictly decreasing as $m\in[t,k-1]$ increases.
 
 \item[\rm(ii)]\;We have $
 	f(m,k,t,n)\leq\spn{n-t}{k-t}$
 for $t\leq m\leq k$, with equality only if $m=t$.
\end{itemize}	
\end{lemma}
\begin{proof}
{\rm(i)}\;Let $m\in[t,k-2]$. Using Lemma \ref{lemmainequntkt}, we obtain
\begin{align*}
	q_m&:=\dfrac{f(m+1,k,t,n)}{f(m,k,t,n)}\leq\dfrac{(k-t+1)(m+1)}{m-t+1}\left(2^{\frac{n-m-1}{k-m-1}}-1\right)^{-1}\\
	&\leq(t+1)(k-t+1)\left(2^{\frac{n-t-1}{k-t-1}}-1\right)^{-1}<1,
\end{align*}
where in the last step we use $n\geq L(k,t)$ to derive  $2^{\frac{L(k,t)-t-1}{k-t-1}}>2(t+1)(k-t+1)$. Therefore, $f(m,k,t,n)$ is strictly decreasing as $m\in[t,k-1]$ increases. Then the other part follows directly, as for $t\leq m\leq k-2$, we have 
	$$\dfrac{(k-t+1)^{m+1-t}\spn{n-m-1}{k-m-1}}{(k-t+1)^{m-t}\spn{n-m}{k-m}}=\dfrac{(m-t+1)q_m}{m+1}<1.$$
	
{\rm(ii)}\;Note that $f(t,k,t,n)=\spn{n-t}{k-t}$. From (i), it remains to prove
	\begin{equation}\label{equlemmamono1}
		\spn{n-t}{k-t}>(k-t+1)^{k-t}\binom{k}{t}.
	\end{equation}
 Using Lemma \ref{lemmalb}, we obtain   $\spn{n-t}{k-t}>\frac{1}{2}(k-t)^{n-k+1}$, and then 
\begin{align*}
	\spn{n-t}{k-t}(k-t+1)^{t-k}\binom{k}{t}^{-1}&>\frac{1}{2}\left(\frac{k-t}{k-t+1}\right)^{k-t}(k-t)^{n-2k+t+1}\binom{k}{t}^{-1}\\
	&>\frac{1}{2e}(k-t)^{L(k,t)-2(k-t)-t+1}\binom{k}{t}^{-1}=\frac{1}{2e}Q(k-t,t).
\end{align*}
Then (\ref{equlemmamono1}) is true as Lemma \ref{lemmaqkt} gives  $Q(k-t,t)\geq Q(2,t)\geq18>2e$.
\end{proof}
\begin{lemma}\label{lemmalemmat+2small}
If $k\geq t+4$ and $n\geq2L(k,t)$, then
$$f(k,k,t,n)<f(t+2,k,t,n)<r(n,k,t).$$
\end{lemma}
\begin{proof}
Recall that $f(t+2,k,t,n)=(k-t+1)^2\binom{t+2}{2}\spn{n-t-2}{k-t-2}$. It follows from Lemma \ref{lemmaspn2lkt} that $f(t+2,k,t,n)<\spn{n-t-1}{k-t-1}$, and thus Lemma \ref{equsizerlb} (i) yields $f(t+2,k,t,n)<r(n,k,t)$. It remains to prove $f(t+2,k,t,n)>f(k,k,t,n)$. By Lemma \ref{lemmalb} we infer
	\begin{align*}
f(t+2,k,t,n)>\frac{1}{2}\binom{t+2}{2}(k-t+1)^2(k-t-2)^{n-k+1}.
	\end{align*}
Since $k\geq t+4$, we have $$\binom{k}{t}^{-1}=\frac{(k-t)(k-t-1)}{k(k-1)}\binom{k-2}{t}^{-1}\geq\frac{12}{(t+4)(t+3)}\binom{k-2}{t}^{-1},$$
and so $\binom{k}{t}^{-1}\binom{t+2}{2}\geq\frac{9}{5}\binom{k-2}{t}^{-1}$. In addition, it is easy to check that $2L(k,t)>2L(k-2,t)>L(k-2,t)+4$. Combining these with Lemma \ref{lemmaqkt}, we obtain that
	\begin{align*}
	\frac{f(t+2,k,t,n)}{f(k,k,t,n)}&>\frac{1}{2}\binom{t+2}{2}\left(\frac{k-t-2}{k-t+1}\right)^{k-t-2}(k-t-2)^{2L(k,t)-2k+t+3}\binom{k}{t}^{-1}\\
	&>\frac{9}{10e^3}(k-t-2)^{L(k-2,t)-2(k-t-2)-t+3}\binom{k-2}{t}^{-1}\\
	&\geq\frac{18}{5e^3}Q(k-t-2,t)\geq\frac{18}{5e^3}Q(2,t)\geq\frac{324}{5e^3}>1.
	\end{align*}
	Thus $f(t+2,k,t,n)>f(k,k,t,n)$.
\end{proof}
\begin{lemma}\label{lemmaqt}
	If $k\geq t+3$ and $n\geq 2L(k,t)$, then $\spn{n-t-s}{k-t-s}+st$ is decreasing as $s\in[2,k-t-1]$ increases.
\end{lemma}
\begin{proof}
We need only to consider the case of $k\geq t+4$. By (\ref{equrecurrence}), it suffices to show that $
(k-t-s)\spn{n-t-s-1}{k-t-s}>t$ for $s\in[2,k-t-2]$. Note that $k-t-s\geq 2$ and $n\geq2L(k,t)>2(t+1)+6(k-t+1)\geq k+t+28$, then $\spn{n-t-s-1}{k-t-s}\geq\spn{n-k+1}{2}=2^{n-k}-1>2^{t+28}-1>100t$. Thus the desired inequality holds.
\end{proof}
\begin{lemma}\label{lemmaspn2lkt'}
Let $k\geq t+3$ and $n\geq2L(k,t)$. Then the following hold.
\begin{itemize}
\item[\rm(i)]$\spn{n-t-s-1}{k-t-s}>(t+1)^2(k-t+1)\spn{n-t-s-1}{k-t-s-1}$ for $0\leq s\leq k-t-2$.
\item[\rm(ii)]$\spn{n-t-3}{k-t-1}>t\spn{n-t-2}{k-t-2}$.
\end{itemize}
\end{lemma}
\begin{proof}
We need an inequality proved in the proof of {\rm\cite[Lemma 19]{Kupavskii-2023}}, that is, for $m\geq\ell\geq2$, 
\begin{equation}\label{equlemmaspn2lkt'}
\spn{m-1}{\ell}\geq\frac{1}{\ell}\left(2^{\frac{m-1}{\ell-1}}-2\right)\spn{m-1}{\ell-1}.
\end{equation}

(i)\;Note that (\ref{equlemmaspn2lkt}) gives $2^{\frac{n-t-s-1}{k-t-s-1}}>(t+1)^2(k-t+1)^2$, then it is routine to check that $\frac{1}{k-t-s}\left(2^{\frac{n-t-s-1}{k-t-s-1}}-2\right)>(t+1)^2(k-t+1)$, and then (i) holds from (\ref{equlemmaspn2lkt'}).

(ii)\;Firstly, by setting $s=2$ in (i), we get  $\spn{n-t-3}{k-t-2}>\spn{n-t-3}{k-t-3}$. It follows from  (\ref{equrecurrence}) that $t\spn{n-t-2}{k-t-2}<t(k-t-1)\spn{n-t-3}{k-t-2}$. On the other hand, by the same argument as presented in the proof of Lemma \ref{lemmaspn2lkt}, we have $2^{\frac{n-t-3}{k-t-2}}>2^{\frac{t}{k-t-1}}(t+1)^2(k-t+1)^2$, and then by (\ref{equlemmaspn2lkt'}) we get   $\spn{n-t-3}{k-t-1}>\frac{(t+1)^2(k-t+1)^2-2}{k-t-1}\spn{n-t-3}{k-t-2}>t^2(k-t+1)\spn{n-t-3}{k-t-2}$. Hence $\spn{n-t-3}{k-t-1}>t\spn{n-t-2}{k-t-2}$.
\end{proof}
\begin{lemma}\label{lemmathm34}
If $k\geq t+3$ and $n\geq2L(k,t)$, then $\max\{u_1(n,k,t), u_2(n,k,t)\}<r(n,k,t).$
\end{lemma}
\begin{proof}
By applying Lemma \ref{lemmaspn2lkt'} (i) to $s=1$, we obtain 
\begin{equation}\label{equlemmathm34}
\spn{n-t-2}{k-t-1}>(t+1)^2(k-t+1)\spn{n-t-2}{k-t-2}.
\end{equation} 
This together with (\ref{equfunr}), (\ref{equsizehlb}) and (\ref{equnkn-1k}) gives
\begin{align*}
	r(n,k,t)-u_1(n,k,t)&\geq\spn{n-t-1}{k-t-1}-\spn{n-t-2}{k-t-1}-\binom{k-t}{2}\spn{n-t-2}{k-t-2}-t\\
	&>\left((k-t-2)\cdot(t+1)^2(k-t+1)-\binom{k-t}{2}\right)\spn{n-t-2}{k-t-2}-t>0.
\end{align*}

Now consider $u_2(n,k,t)$. Set $h:=\spn{n-t-2}{k-t-1}-\spn{n-t-3}{k-t-1}-(k-t-1)\spn{n-t-2}{k-t-2}-t$ for short. Then from $\spn{n-t-3}{k-t-1}<\frac{1}{k-t-1}\spn{n-t-2}{k-t-1}$, (\ref{equlemmathm34}) and Lemma \ref{lemmaspn2lkt}, we derive   $$h>t(k-t)\spn{n-t-2}{k-t-2}>t^3(k-t)^3\spn{n-t-3}{k-t-3}.$$
 Hence $r(n,k,t)-u_2(n,k,t)\geq h-\binom{k-t-1}{3}\spn{n-t-3}{k-t-3}>0$, as desired.
\end{proof}
\begin{lemma}\label{lemmahh1v}
	Let $k\geq t+3$ and $n\geq2L(k,t)$. The following hold.
	\begin{itemize}
		\item[\rm(i)]We have $r(n,k,t)\leq|\mathcal{H}_1(n,k,t)|<|\mathcal{H}(n,k,t)|.$
		\item[\rm(ii)]If $k\geq2t+3$, then $|\mathcal{A}(n,k,t)|<r(n,k,t)$.
	\end{itemize}
\end{lemma}
\begin{proof}
  Set $d=d(n,k,t):=\spn{n-t-1}{k-t-1}-\spn{n-t-2}{k-t-1}$ for short. Note that $r(n,k,t)=|\mathcal{H}(n,k,t)|-d$.
	
	(i)\;By inclusion-exclusion, $\mathcal{H}_1(n,k,t)$ has size
	\begin{align*}
		\sum_{j=1}^{k-t}(-1)^{j-1}\left(\binom{k-t-1}{j}\spn{n-t-j}{k-t-j}+\binom{k-t-1}{j-1}\spn{n-t-j-1}{k-t-j}\right)+t.
	\end{align*}
	By Lemma \ref{lemmain-ex}, we have  $|\mathcal{H}_1(n,k,t)|\leq(k-t-1)\spn{n-t-1}{k-t-1}+\spn{n-t-2}{k-t-1}+t$. Then (\ref{equsizehlb}), Lemma \ref{lemmaspn2lkt}, (\ref{equnkn-1k}) and $k\geq t+3$ lead to
	$$|\mathcal{H}_1(n,k,t)|-|\mathcal{H}(n,k,t)|<-\left(1-\frac{1}{(t+1)^2}-\frac{1}{k-t-1}\right)\spn{n-t-1}{k-t-1}<0.$$
	
	Now we prove $r(n,k,t)\leq|\mathcal{H}_1(n,k,t)|$. Write $D_1:=|\mathcal{H}_1(n,k,t)|+d-|\mathcal{H}(n,k,t)|$ for short. By Lemma \ref{lemmain-ex}, $|\mathcal{H}_1(n,k,t)|+d$ is at least 
	\begin{align*}
	(k-t)\spn{n-t-1}{k-t-1}-\binom{k-t-1}{2}\spn{n-t-2}{k-t-2}-(k-t-1)\spn{n-t-3}{k-t-2}+t.
	\end{align*}
	 From (\ref{equsizeh}) and Lemma \ref{lemmain-ex}, we derive
	\begin{align}
		|\mathcal{H}(n,k,t)|\leq&\sum_{j=1}^{3}(-1)^{j-1}\binom{k-t}{j}\spn{n-t-j}{k-t-j}+t.\label{equsizehub}
	\end{align}
It follows from (\ref{equnkn-1k}) and Lemma \ref{lemmaspn2lkt} that
	\begin{align*}
		D_1&\geq(k-t-1)\spn{n-t-2}{k-t-2}-(k-t-1)\spn{n-t-3}{k-t-2}-\binom{k-t}{3}\spn{n-t-3}{k-t-3}\\
		&\geq(k-t-3)\spn{n-t-2}{k-t-2}-\binom{k-t}{3}\spn{n-t-3}{k-t-3}\geq0.
	\end{align*}
	Thus $|\mathcal{H}_1(n,k,t)|\geq r(n,k,t)$.	
	
	(ii)\;Write $D_2 :=|\mathcal{A}(n,k,t)|+d-|\mathcal{H}(n,k,t)|$ for short. Note that $|\mathcal{A}(n,k,t)|+d<(t+3)\spn{n-t-1}{k-t-1}-\spn{n-t-2}{k-t-1}$, then from (\ref{equsizehlb}) and (\ref{equlemmathm34}), we obtain 
$$D_2<-(k-2t-3)\spn{n-t-1}{k-t-1}+\left(\binom{k-t}{2}-(t+1)^2(k-t+1)\right)\spn{n-t-2}{k-t-2}.$$
If $k=2t+3$, then clearly $D_2<0$. For the case of $k\geq2t+4$, we deduce $D_2<0$ from Lemma \ref{lemmaspn2lkt}. Thus $|\mathcal{A}(n,k,t)|<r(n,k,t)$.
\end{proof}
 \begin{lemma}\label{lemmawholargest}
	Let $k\geq t+3$ and $n\geq2L(k,t)$. The following hold.
	\begin{itemize}
		\item[\rm(i)]If $k\geq2t+3$, then $|\mathcal{H}(n,k,t)|>|\mathcal{A}(n,k,t)|.$
		\item[\rm(ii)]If $(k,t)=(4,1)$, then $|\mathcal{H}(n,k,t)|=|\mathcal{A}(n,k,t)|$.
		\item[\rm(iii)]If $k\leq2t+2$ and $(k,t)\neq(4,1)$, then $|\mathcal{A}(n,k,t)|>|\mathcal{H}(n,k,t)|.$
	\end{itemize}  
\end{lemma}
\begin{proof}
	(i)\;This follows directly from Lemma \ref{lemmahh1v}.
	
	(ii)\;It is routine to check that $|\mathcal{H}(n,4,1)|=3\spn{n-2}{2}-2=|\mathcal{A}(n,4,1)|.$
	
	(iii)\;Note that $|\mathcal{H}(n,k,t)|<(k-t)\spn{n-t-1}{k-t-1}+t$ from (\ref{equsizeh}) and  Lemma \ref{lemmain-ex}. Hence, if $k\leq2t+1$, then Lemma \ref{lemmaspn2lkt} gives
	\begin{align*}
		|\mathcal{A}(n,k,t)|-|\mathcal{H}(n,k,t)|&>(2t+2-k)\spn{n-t-1}{k-t-1}-(t+1)\spn{n-t-2}{k-t-2}-t>0.
	\end{align*}
	
	Suppose that $k=2t+2$ and $t\geq2$. We have from Lemma \ref{lemmaspn2lkt}, $t(t+2)<(t+1)^2$ and $k-t+1>4$ that
	\begin{align*}
		\binom{t+2}{3}\spn{n-t-3}{k-t-3}&<\frac{t(t+2)}{6(t+1)(k-t+1)^2}\spn{n-t-2}{k-t-2}<\frac{t+1}{96}\spn{n-t-2}{k-t-2}.
	\end{align*}
	This together with (\ref{equsizea}) and (\ref{equsizehub}) yields
	\begin{align*}
		|\mathcal{A}(n,k,t)|-|\mathcal{H}(n,k,t)|&\geq\binom{t+1}{2}\spn{n-t-2}{k-t-2}-\binom{t+2}{3}\spn{n-t-3}{k-t-3}-t\\
		&>\frac{(t+1)(48t-1)}{96}\spn{n-t-2}{k-t-2}-t>0,
	\end{align*}
	where in the last step we use $t\geq2$. 
\end{proof}
\section*{Acknowledgments}
B. Lv is supported by National Natural Science Foundation of China (12571347 \& 12131011), and Beijing Natural Science Foundation (1252010).
\
\addcontentsline{toc}{chapter}{Bibliography}

{
	}
\end{document}